\documentclass[11pt]{article}
\usepackage{amsmath, amssymb, amscd, amsthm, amsfonts}
\usepackage{graphicx}
\usepackage{hyperref}
\usepackage{amssymb}
\usepackage{amsmath}
\usepackage{amsthm}
\usepackage{amsfonts}

\usepackage{a4wide}
\usepackage{longtable}
\usepackage{multirow}

\usepackage[latin1]{inputenc}
\usepackage[english]{babel}
\usepackage[T1]{fontenc}
\usepackage{amssymb,amsmath,amsthm,amstext,amsfonts,latexsym}

\usepackage{pgf,tikz}
\usepackage{mathrsfs}
\usetikzlibrary{arrows}

\newtheorem{lemma}{Lemma}[section]
\newtheorem{proposition}{Proposition}[section]
\newtheorem{Properties}{Properties}[section]

\newtheorem{theorem}{Theorem}[section]

\newtheorem{remark}{Remark}[section]
\theoremstyle{definition}

\theoremstyle{remark}

\usepackage[all]{xy}
\title{On the exact divisibility by $5$ of the class number of\\ some pure metacyclic fields}
\date{}

\begin{document}

\maketitle

\begin{center}
{\sc Fouad ELMOUHIB }\\
{\footnotesize Department of Mathematics and Computer Sciences,\\
Mohammed first University, Oujda, Morocco,\\
Correspondence: fouad.cd@gmail.com}\\
\vspace{0.7cm}
{\sc Mohamed TALBI }\\
{\footnotesize Regional Center of Professions of Education and Training,\\
ksirat1971@gmail.com}\\

\vspace{0.7cm}
{\sc Abdelmalek AZIZI }\\
{\footnotesize Department of Mathematics and Computer Sciences,\\
Mohammed first University, Oujda, Morocco,\\
abdelmalekazizi@yahoo.fr}
\end{center}

\begin{abstract}
Let $\Gamma \,=\, \mathbb{Q}(\sqrt[5]{n})$ be a pure quintic field, where $n$ is a natural number $5^{th}$ power-free. Let $k = \mathbb{Q}(\sqrt[5]{n}, \zeta_5)$, with $\zeta_5$ is a primitive  $5^{th}$ root of unit, be the normal closure of $\Gamma$, and  a pure metacyclic field of degree $20$ over $\mathbb{Q}$. When $n$ takes some particular forms, we show that $\Gamma$ admits a trivial $5$-class group and $5$ divides exactly the class number of $k$.
\end{abstract}

\section{Introduction}
In algebraic number theory, the task to determine the class number of a number field $k$ is until nowday a classical and difficult open problem, although we have explicit formulas and relations of class numbers.\\
In the case when this task is hard to deal with it, its seems satisfactory to determine the exact power of a prime $p$, which divides the class number of $k$.\\
Let $\Gamma = \mathbb{Q}(\sqrt[5]{n})$ be a pure quintic field, where $n$ is a $5^{th}$ power-free natural greater than one, and $k_0 = \mathbb{Q}(\zeta_5)$ be the $5^{th}$ cyclotomic field. Then $k = \Gamma(\zeta_5)$ is the normal closure of $\Gamma$ and a pure metacyclic field. There is a question when the class numbers $h_\Gamma$ of $\Gamma$ and $h_k$ of $k$ are divisible by $5$. This divisibility was studied by several researchers until now. Subsequently to Honda's study of the cubic case \cite{Honda}, C.Parry \cite{Pa} studied with difficulty this question. He presented the relation between $h_\Gamma$ and $h_k$:\\
\centerline{$5^5h_k = uh_\Gamma^4$}\\
where $u$ is a divisor of $5^6$. He also proved that $h_\Gamma$ is divisible by $5$ if and only if $h_k$ is divisible by $5^2$. The authors of \cite{Mani} gave a several results on the rank of the $5$-Sylow subgroup of the ideal class group of $\Gamma$ and $k$, by means of tools of genus theory and Kummer duality.\\
Due to an investigation of the theory developed in \cite{Mani}, we gave in \cite{FOU1} a full classification of all normal closures $k$ whenever its $5$-class group is of type $(5, 5)$, and in order to illuminate this classification by numerical examples obtained by PARI/GP \cite{PRI}, we have noticed that for large values of some forms of $n$, the $5$-class group of $k$ is cyclic of order $5$, which means that $h_k$ is exactly divisible by $5$. The done calculations allow us to conjecture that this divisibility is verified for these forms of $n$, see [\cite{FOU1}, conjecture 4.1].\\
In this paper, we are interested in giving a proof of this conjecture. In fact we shall prove the following main theorem:

\begin{theorem} \label{theo 1.1}
Let $\Gamma\,=\,\mathbb{Q}(\sqrt[5]{n})$ be a pure quintic field, where $n$ is a natural $5^{th}$ power-free. Let $k$ be the normal closure of $\Gamma$. Denote by $h_{\Gamma, 5}$ and $h_{k, 5}$ the $5$-class number respectively of $\Gamma$ and $k$. Let $q_1$ and $q_2$ be primes such that $q_1, q_2\,\equiv \, \pm2\, (\mathrm{mod}\, 5)$. If the natural $n$ takes one forms as follows:
\begin{equation}
 n\,=\,\left\lbrace
   \begin{array}{ll}
   
     q_1^{e_1}q_2 & \text{ with } \quad q_i\,\equiv \,\pm7\, (\mathrm{mod}\, 25)\\

     5q_1  & \text{ with }\quad q_1\,\equiv \,\pm7\, (\mathrm{mod}\, 25)\\
   
     5q_1q_2 & \text{ with } \quad q_1 \,\text{ or }\, q_2$  $\not\equiv \,\pm7\, (\mathrm{mod}\, 25)\\

   \end{array}
   \right.
\end{equation}
and $q_2$ and $5$ are not a quintic residue modulo $q_1$, then $h_{\Gamma, 5}$ is trivial and $h_{k, 5} = 5$.
\end{theorem}

This result will be underpinned by numerical examples obtained with the computational number theory system PARI/GP [\cite{PRI}].
 
\section{Norm residue symbol}
Let $L/K$ an abelian extension of number fields with conductor $f$. For each finite or infinite prime ideal $\mathcal{P}$ of $K$, we note by $f_{\mathcal{P}}$ the largest power of $\mathcal{P}$ that divides $f$. Let
$\beta \in K^{*}$, we determine an auxiliary number $\beta_0$ by the two conditions $\beta_0\equiv\beta\,(\mathrm{mod}\, f_\mathcal{P})$ and $\beta_0\equiv1\,(\mathrm{mod}\, \frac{f}{f_\mathcal{P}})$. Let $\mathcal{Q}$ an ideal co-prime with $\mathcal{P}$ such that $(\beta_0) = \mathcal{P}^a\mathcal{Q}$ ($a=0$ if $\mathcal{P}$ infinite). We note by
\begin{center}
\large{ $\left(\frac{\beta,L}{\mathcal{P}} \right) =  \left(\frac{L/K}{\mathcal{Q}} \right)$}
\end{center}
the Artin map in $L/K$ applied to $\mathcal{Q}$.\\
Let $K$ be a number field containing the $m^{th}$-roots of units, where $m\in \mathbb{N}$, then for each $\alpha,\beta \in K^{*}$ and prime ideal $\mathcal{P}$ of $K$, we define the norm residue symbol by:
\begin{center}
\large{ $\left(\frac{\beta,\alpha}{\mathcal{P}} \right)_m =  \frac{\left(\frac{\beta,K(\sqrt[m]{\alpha})}{\mathcal{P}} \right)\sqrt[m]{\alpha}}{\sqrt[m]{\alpha}}$}
\end{center}
Therefore, if the prime ideal $\mathcal{P}$ is unramified in the field $K(\sqrt[m]{\alpha})$, then we write:
\begin{center}
\large{ $\left(\frac{\alpha}{\mathcal{P}} \right)_m =  \frac{\left(\frac{K(\sqrt[m]{\alpha})}{\mathcal{P}} \right)\sqrt[m]{\alpha}}{\sqrt[m]{\alpha}}$}
\end{center}
\begin{remark}
\large{Notice that $\left(\frac{\beta,\alpha}{\mathcal{P}} \right)_m$ and $\left(\frac{\alpha}{\mathcal{P}} \right)_m$} are two $m^{th}$-roots of unit.
\end{remark}
\large{Following \cite{Hass}, the principal properties of the norm residue symbol are given as follows:}
\begin{Properties}\label{normprop}
\large{\item[$(1)$] $\left(\frac{\beta_1\beta_2,\alpha}{\mathcal{P}} \right)_m = \left(\frac{\beta_1,\alpha}{\mathcal{P}} \right)_m\left(\frac{\beta_2,\alpha}{\mathcal{P}} \right)_m$};

\large{\item[$(2)$] $\left(\frac{\beta,\alpha_1\alpha_2}{\mathcal{P}} \right)_m = \left(\frac{\beta,\alpha_1}{\mathcal{P}} \right)_m\left(\frac{\beta,\alpha_2}{\mathcal{P}} \right)_m$};

\large{\item[$(3)$] $\left(\frac{\beta,\alpha}{\mathcal{P}} \right)_m = \left(\frac{\alpha,\beta}{\mathcal{P}} \right)_m^{-1}$};

\item[$(4)$] If $\mathcal{P}$ is not divisible by the conductor $f(\sqrt[m]{\alpha})$ of $K(\sqrt[m]{\alpha})$ and appears in $(\beta)$ with the exponent b, then: 
\large{$\left(\frac{\beta,\alpha}{\mathcal{P}} \right)_m = \left(\frac{\alpha}{\mathcal{P}} \right)_m^{-b}$  };

\large{\item[$(5)$] $\left(\frac{\beta,\alpha}{\mathcal{P}} \right)_m = 1$ if and only if $\beta$ is norm residue of $K(\sqrt[m]{\alpha})$ modulo $f(\sqrt[m]{\alpha})$  };

\large{\item[$(6)$] $\left(\frac{\tau\beta,\tau\alpha}{\tau\mathcal{P}} \right)_m = \tau\left(\frac{\beta,\alpha}{\mathcal{P}} \right)_m$ for each automorphism $\tau$ of $K$ };

\large{\item[$(7)$] ${\displaystyle \prod_{\mathcal{P}} \left(\frac{\beta,\alpha}{\mathcal{P}} \right)_m} = 1$} for all finite or infinite prime ideals;

\item[$(8)$]If $K'$ is a finite extension of $K$, $\alpha \in K^{*},\beta' \in K'$ then: 
\begin{center}
\large{${\displaystyle \prod_{\mathcal{P'}|\mathcal{P}} \left(\frac{\beta',\alpha}{\mathcal{P'}} \right)_m} =  \left(\frac{\mathcal{N}_{K'/K}(\beta'),\alpha}{\mathcal{P}} \right)_m$}
\end{center}

\item[$(9)$]Let $\alpha,\beta \in K^{*}$ and the conductors $f(\sqrt[m]{\alpha})$, $f(\sqrt[m]{\beta})$ of respectively $K(\sqrt[m]{\alpha})$, $K(\sqrt[m]{\beta})$ are co-prime then, the classical reciprocity law:
\begin{center}
\large{$\left(\frac{\beta}{(\alpha)} \right)_m = \left(\frac{\alpha}{(\beta)} \right)_m$}
\end{center}

%\end{itemize}
\end{Properties}

\large{For more basic properties of the norm residue symbol in the number fields, we refer the
reader to \cite{Hass}.\\
Notice that in the rest of the article, we will use the norm quintic residue symbols $(m = 5)$. If we deal with a principal ring of integer, we will write the norm quintic residue symbol as follows:}
\begin{center}
\large{$\left(\frac{\beta,\alpha}{(\pi)} \right)_5 = \left(\frac{\beta,\alpha}{\pi} \right)_5$ and $\left(\frac{\alpha}{(\pi)} \right)_5 = \left(\frac{\alpha}{\pi} \right)_5$}
\end{center}

\section{Ambiguous ideal classes}
Let $\Gamma = \mathbb{Q}(\sqrt[5]{n})$ be a pure quintic field, where $n$ is a natural $5^{th}$ power-free, and $k_0 = \mathbb{Q}(\zeta_5)$ be the cyclotomic field generated by a primitive $5^{th}$ root of unit $\zeta_5$. Then $k = \Gamma(\zeta_5)$ is the normal closure of $\Gamma$, and a cyclic Kummer extension of degree $5$ of $k_0$. By $C_{k, 5}$ we denote the $5$-ideal class group of $k$, and by $C_{k, 5}^{(\sigma)}$ the subgroup of ambiguous ideal classes under the action of $Gal(k/k_0) = \langle \sigma \rangle$.\\
In [\cite{Hass2}, Theorem 13], Hasse specifies rank $C_{k, 5}^{(\sigma)}$ as follows:\\
\centerline{rank $C_{k, 5}^{(\sigma)}$ = $d+q^*-(r+1+o)$}\\
where\\
$\bullet$ $d =$ number of ramified primes in $k/k_0$.\\
$\bullet$ $r =$ rank of the free abelian part of the group of units $E_{k_0}$ of $k_0$, so $r = 1$.\\
$\bullet$ $o = 1$ because $k_0$ contains a primitive $5^{th}$ root of unit.\\
$\bullet$ $q^*$ is defined by $[N_{k/k_0}(k-\{0\})\cap E_{k_0} : N_{k/k_0}(E_{k_0})] = 5^{q^*}$. Here $N_{k/k_0}$ is the relative norm from $k$ to $k_0$.\\
We obtain that  rank $C_{k, 5}^{(\sigma)}$ = $d+q^*-3$.\\
We note that $ N_{k/k_0}(E_{k_0}) = E_{k_0}^5$ and $[E_{k_0} : E_{k_0}^5] = 5^2$, so we get that $q^* \in \{0, 1, 2\}$.\\
The group $E_{k_0}$ of units is generated by $\zeta_5$ and $\zeta_5+1$, then according to the definition of $q^*$, we see that:
\begin{center}
$q^*$ = $ \begin{cases}
2 & \text{if }\, \zeta, \zeta+1 \in N_{k/k_0}(k^*),\\
1 & \text{if }\, \zeta^i(\zeta+1)^j \in N_{k/k_0}(k^*)\,\text{ for some i and j },\\
0 & \text{if }\, \zeta^i(\zeta+1)^j \notin N_{k/k_0}(k^*)\, \text{for}\hspace{2mm} 0\leq i,j\leq 4 \text{ and}\hspace{2mm} i+j\neq 0.\\
\end{cases}$
\end{center}
The following lemma gives us some results, which allow us to determine the value of $q^*$.

\begin{lemma}\label{lem 3.1}
Let $k_0 = \mathbb{Q}(\zeta_5)$ and let $k =k_0(\sqrt[5]{n})$, where $n = u\lambda^{e_\lambda}\pi_1^{e_1}....\pi_g^{e_g}$, with $u$ is a unit in $\mathbb{Z}[\zeta_5]$, $\lambda = 1-\zeta_5$ is the unique prime over $5$, and $\pi_1,...,\pi_g$ prime elements in $\mathbb{Z}[\zeta_5]$. Then:\\
$\bullet$ $\zeta_5\, \in N_{k/k_0}(k-\{0\})\, \Longleftrightarrow\, N_{k_0/\mathbb{Q}}((\pi_i))\, \, \equiv \, 1 \,(\mathrm{mod}\,25)$ for all i.\\
$\bullet$ $\zeta_5^i(\zeta_5+1)^j\, \in N_{k/k_0}(k-\{0\})\, \Longleftrightarrow\,$  
every $\pi | x$ above has the property that $\zeta_5^i(\zeta_5+1)^j$ is a $5^{th}$ power modulo $(\pi)$ in $\mathbb{Z}[\zeta_5]$ for all $i, j$.\\
$\bullet$ $(\lambda)$ ramifies in $k/k_0$  $\Longleftrightarrow\,$ $n \not\equiv \, \pm1, \pm7 \,(\mathrm{mod}\,\lambda^5)$.\\
\end{lemma}

\begin{proof}
See [\cite{Mani}, Lemma 5.1]
\end{proof}

The order of the subgroup $C_{k, 5}^{(\sigma)}$ of ambiguous ideal classes whenever $n$ takes forms of theorem \ref{theo 1.1} is given by the following proposition:
\begin{proposition}\label{prop 3.1}
Let $\Gamma$, $k_0$ and $k$ as above. Let $q_1$, $q_2$ a prime numbers such that\\ $q_1 , q_2 \equiv \pm2 \,(\mathrm{mod}\, 5)$. If $n$ takes one of the following forms
\begin{equation}
 n\,=\,\left\lbrace
   \begin{array}{ll}
   
     q_1^{e_1}q_2 & \text{ with } \quad q_i\,\equiv \,\pm7\, (\mathrm{mod}\, 25)\\

     5q_1  & \text{ with }\quad q_1\,\equiv \,\pm7\, (\mathrm{mod}\, 25)\\
   
     5q_1q_2 & \text{ with } \quad q_1 \,\text{ or }\, q_2$  $\not\equiv \,\pm7\, (\mathrm{mod}\, 25)\\

   \end{array}
   \right.
\end{equation}
Then the subgroup $C_{k, 5}^{(\sigma)}$ of ambiguous ideal classes is cyclic of order $5$.
\end{proposition}

\begin{proof}
We will successively treat the three forms to calculate $d$ and $q^*$ defined before, in order to show that rank $C_{k, 5}^{(\sigma)} = 1$.\\
We note that if $q$ is a prime of $\mathbb{Z}$ such that $q \equiv \pm2 \,(\mathrm{mod}\, 5)$, then by [\cite{washint}, Theorem 2.13], $q$ is remain inert in $k_0$ and $N_{k_0/\mathbb{Q}}(q) = q^4$.\\
Since $k$ is the normal closure of $\Gamma$, then $k = k_0(\sqrt[5]{n})$. We see that if a prime $p$ of $\mathbb{Z}$ is ramified in $\Gamma$, then all primes of $k_0$ above $p$ are ramified in $k$. Since in proposition \ref{prop 3.1} we deal with primes $q_i \equiv \pm2 \,(\mathrm{mod}\, 5)$, $(i = 1, 2)$, which are inert in $k_0$, so if $q_i$ are ramified in $\Gamma$, then they are ramified in $k$ too.\\

$\bullet$ If $n = q_1q_2$ such that $q_i \equiv \pm7 \,(\mathrm{mod}\, 25)$, then $n \equiv \pm1 \,(\mathrm{mod}\, 25)$ and by lemma \ref{lem 3.1}, we have that $\lambda$ is not ramified in $k/k_0$. The primes $q_i$ are ramified in $k$, because $disk(\Gamma/\mathbb{Q}) = 5^5n^4$ and $q_i$ divide this discriminant. Hence we have $d = 2$.\\
Since $q_i \equiv \pm7 \,(\mathrm{mod}\, 25)$, then $N_{k/\mathbb{Q}}(q) = q^4 \equiv 1 \,(\mathrm{mod}\, 25)$ and by lemma \ref{lem 3.1} we have $\zeta_5\, \in N_{k/k_0}(k-\{0\})$. By the same reasoning we have that $\zeta_5+1\, \in N_{k/k_0}(k-\{0\})$, so $q^* = 2$, whence rank $C_{k, 5}^{(\sigma)} = 1$.

$\bullet$ If $n = 5q_1$ such that $q_1 \equiv \pm7 \,(\mathrm{mod}\, 25)$, we have $n \not\equiv \pm1,\pm7 \,(\mathrm{mod}\, 25)$, so $\lambda$ is ramified in $k/k_0$. As before we have $q_1$ is ramified also in $k/k_0$, whence $d = 2$.\\
Since $q_1 \equiv \pm7 \,(\mathrm{mod}\, 25)$, we proceed as the previous point to prove that $q^* = 2$. Hence rank $C_{k, 5}^{(\sigma)} = 1$.

$\bullet$ If $n = 5q_1q_2$ such that $q_1$ or $ q_2 \not\equiv \pm7 \,(\mathrm{mod}\, 25)$. We see that $\lambda, q_1$ and $q_2$ are ramified in $k/k_0$ so $d = 3$. Since $q_1$ or $ q_2 \not\equiv \pm7 \,(\mathrm{mod}\, 25)$, and by lemma \ref{lem 3.1} we have that $\zeta_5\, \not\in N_{k/k_0}(k-\{0\})$, which imply that $q^* < 2$. Precisely, according to the proof of [\cite{Mani}, Theorem 5.13], we have $q^* = 1$. Hence rank $C_{k, 5}^{(\sigma)} = 1$.\\
We proved that for the three forms of the natural $n$, we have rank $C_{k, 5}^{(\sigma)} = 1$, which means that $C_{k, 5}^{(\sigma)}$ is a cyclic subgroup of $C_{k, 5}$ of order $5$.
\end{proof}

\section{Proof of main theorem}
Let $k$ be the normal closure of a pure quintic field $\Gamma = \mathbb{Q}(\sqrt[5]{n})$, where $n$ takes the forms mentioned above, also $k$ is cyclic Kummer extension of degree $5$ of the cyclotomic field\\ $k_0 = \mathbb{Q}(\zeta_5)$. We put $Gal(k/k_0) = \langle \sigma \rangle$. By $k_5^{(1)}$ we denote the Hilbert $5$-class field of $k$, that is the maximal abelian unramified extension of $k$ of degree a power of $5$. By class field theory $C_{k, 5} \simeq Gal(k_5^{(1)}/k)$.\\
Next we define the genus field of $k/k_0$, which we denote by $k^*$, to be the maximal abelian extension of $k_0$ contained in $k_5^{(1)}$. Then using the isomorphism $C_{k, 5} \simeq Gal(k_5^{(1)}/k)$, we see that $Gal(k_5^{(1)}/k^*)$ can be identified with a subgroup of $C_{k, 5}$, which is called the principal genus. By class field theory this subgroup is $C_{k, 5}^{1-\sigma}$. Its easy to see that $Gal(k^{*}/k) \simeq C_{k, 5}/C_{k, 5}^{1-\sigma}$, and we have the following lemma:
\begin{lemma}\label{lem 4.1}
Let $k, C_{k, 5}, C_{k, 5}^{(\sigma)}$ and $C_{k, 5}^{1-\sigma}$ as above. Then we have:\\
\centerline{rank $C_{k, 5}^{(\sigma)} = $ rank $C_{k, 5}/C_{k, 5}^{1-\sigma}$ = $1$}
\end{lemma}

\begin{proof}
Since the $5^{th}$ cyclotomic field $k_0$, has a trivial class group, we can admit the same proof of [\cite{GER1}, Lemma 2.3], and by proposition \ref{prop 3.1}, we have rank $C_{k, 5}^{(\sigma)} = 1$ 
\end{proof}
From lemma \ref{lem 4.1}, we deduce that rank $Gal(k^{*}/k) =$ rank $C_{k, 5}^{(\sigma)} = 1$, which means that $[k^* : k] = 5$. By Kummer theory, there exist $x_1 \in k$ such that $k^* = k(\sqrt[5]{x_1})$.\\
In [\cite{Mani}, section 6], we find an investigation of the $5$-class group of pure quintic field $\Gamma$, by giving a upper bound of its rank as follows:\\
\centerline {rank $C_{\Gamma, 5} \leq min\{t,\, t-s + rank (C_{k, 5}/C_{k, 5}^{1-\sigma})^+\}$},
where $t$ is the rank of $C_{k, 5}^{(\sigma)}$, $s = \mathrm{rank} (C_{k, 5}^{(\sigma)}.C_{k, 5}^{1-\sigma})/C_{k, 5}^{1-\sigma}$ and $(C_{k, 5}/C_{k, 5}^{1-\sigma})^+$ is defined as [\cite{Mani}, Lemma 6.1].\\
We note that according to [\cite{Mani}, Theorem 6.6], if the natural $n$ is not divisible by any prime $p \,\equiv\, 1\, (\mathrm{mod}\, 5)$, we have that $rank (C_{k, 5}/C_{k, 5}^{1-\sigma})^+ = 0$, which is verified in our situation, whence the upper bound of rank $C_{\Gamma, 5}$ becomes $t-s$.\\
The following theorem allows us to compute the value of $s$ in terms of rank of matrix with entries in $\mathbb{F}_5$ the finite field of $5$ elements.

\begin{theorem}\label{theo 4.1}
Let $k_0 = \mathbb{Q}(\zeta_5)$ and $k = k_0(\sqrt[5]{n})$, where $n$ takes forms as above and decompose in $k_0$ as $n = u\lambda^{e_\lambda}\pi_1^{e_1}...\pi_g^{e_g}$, with $u$ is a unit, $\lambda = 1-\zeta_5$ the unique prime in $k_0$ above $5$, $\pi_i $ primes of $k_0$, $e_\lambda \in \{0, 1, 2, 3, 4\}$ and $e_i \in \{1, 2, 3, 4\}$. Let $k^* = k(\sqrt[5]{x_1})$ be the genus field of $k/k_0$.\\
Let $\alpha_{1j} \in \mathbb{F}_5$ such that
\begin{center}
$\zeta_5^{\alpha_{1j}}$ = \large{$\left(\frac{x_1, n}{\pi_j} \right)_5$} for $1 \leq j \leq g$\\
\vspace{0.4cm}
\hspace{1.6cm} $\zeta_5^{\alpha_{1(g+1)}}$ = \large{$\left(\frac{x_1, \lambda}{\lambda} \right)_5$} if $\lambda$ is ramified in $k/k_0$.\\
\end{center}
If $M$ is the matrix $(\alpha_{1j})$, then $s = \mathrm{rank}\, M$.
\end{theorem}

\begin{proof}
See [\cite{Mani}, Theorem 5.10] for $t = 1$
\end{proof}

In the remainder, we compute the value of $s$ for the three forms of the natural $n$.\\
By the definition of the matrix $M$ in theorem \ref{theo 4.1}, we can see that $s = \mathrm{rank}\, M = 1$ if and only if one $\alpha_{1j} \neq 0$.

$\bullet$ If $n = q_1q_2$ such that $q_i \equiv \pm7 \,(\mathrm{mod}\, 25)$, $(i = 1, 2)$, According to [\cite{Limura}, Lemma 3.3], we have that $k^* = k(\sqrt[5]{q_1})$, so we put $x_1 = q_1$ and since $q_i$, $(i = 1, 2)$, are inert in $k_0$, we put $\pi_1 = q_1$, then we can get the value of $\alpha_{11}$ by calculus of \large{$\left(\frac{x_1, n}{\\pi_1} \right)_5$} $=$ \large{$\left(\frac{q_1,\, q_1q_2}{q_1} \right)_5$}. We have\\
\large{$\left(\frac{q_1,\, q_1q_2}{q_1} \right)_5$} = \large{$\left(\frac{q_1, q_1}{q_1} \right)_5$}\large{$\left(\frac{q_1, q_2}{q_1} \right)_5$} by $(1)$, $(2)$ of properties \ref{normprop}.\\\\
\vspace{0.5cm}
and since\\
\vspace{0.5cm}
\large{$\left(\frac{q_1, q_1}{q_1} \right)_5$} = $1$ because $q_1$ is norm in $k_0(\sqrt[5]{q_1})/k_0$.\\
\large{$\left(\frac{q_1, q_2}{q_1} \right)_5$} = \large{$\left(\frac{q_2}{q_1} \right)_5^{-1}$} by $(4)$ of properties \ref{normprop}.\\
we deduce that $\zeta_5^{\alpha_{11}}$ = \large{$\left(\frac{q_2}{q_1} \right)_5^{-1}$}  $\neq 1$ because $q_2$ is not a quintic residue modulo $q_1$, hence $\alpha_{11} \neq 0$, which imply that $s = 1$.

$\bullet$ If $n = 5q_1$ such that $q_1 \equiv \pm7 \,(\mathrm{mod}\, 25)$, by [\cite{Limura}, Lemma 3.3], we have that $k^* = k(\sqrt[5]{q_1})$, so we admit the same reasoning as the previous point. Its sufficient to replace $q_2$ by $5$. Hence we have $s= 1$.

$\bullet$ If $n = 5q_1q_2$ such that $q_1$ or $q_2 \not\equiv \pm7 \,(\mathrm{mod}\, 25)$, and without loosing generality we can assume that $q_2 \not\equiv \pm7 \,(\mathrm{mod}\, 25)$. We have that $k^* = k(\sqrt[5]{5q_2})$. Put $x_1 = 5q_2$ and $\pi_1 = q_1$, then we calculate  \large{$\left(\frac{x_1, n}{\\pi_1} \right)_5$} $=$ \large{$\left(\frac{5q_2,\, 5q_1q_2}{q_1} \right)_5$}. We have\\
\large{$\left(\frac{5q_2,\, 5q_1q_2}{q_1} \right)_5$} = \large{$\left(\frac{5, 5}{q_1} \right)_5$}\large{$\left(\frac{5, q_1}{q_1} \right)_5$}\large{$\left(\frac{5, q_2}{q_1} \right)_5$}\large{$\left(\frac{q_2, 5}{q_1} \right)_5$}\large{$\left(\frac{q_2, q_1}{q_1} \right)_5$}\large{$\left(\frac{q_2, q_2}{q_1} \right)_5$} by $(1)$, $(2)$ of properties \ref{normprop}.\\\\
\vspace{0.5cm}
and since\\
\vspace{0.5cm}
\large{$\left(\frac{5, 5}{q_1} \right)_5$} = $1$ by $(2)$ and $(4)$ of properties \ref{normprop}.\\
\large{$\left(\frac{5, q_1}{q_1} \right)_5$} = \large{$\left(\frac{5}{q_1} \right)_5$} $\neq 1$ by $(4)$ of properties \ref{normprop}.\\
\large{$\left(\frac{5, q_2}{q_1} \right)_5$}\large{$\left(\frac{q_2, 5}{q_1} \right)_5$} $= 1$ by $(3)$ of properties \ref{normprop}.\\
\large{$\left(\frac{q_2, q_1}{q_1} \right)_5$} = \large{$\left(\frac{q_2}{q_1} \right)_5$} $\neq 1$ by $(4)$ of properties \ref{normprop}.\\
\large{$\left(\frac{q_2, q_2}{q_1} \right)_5$} = $1$ by $(2)$ and $(4)$ of properties \ref{normprop}.\\
We deduce that $\zeta_5^{\alpha_{11}} \neq 1$, because $q_2$ and $5$ are not a quintic residue modulo $q_1$, which imply that $s = 1$.\\

In summary, we proved that for the three forms of the natural $n$ we have $s = 1$, and by proposition \ref{prop 3.1}, we have $t = 1$, whence the upper bound of rank $C_{\Gamma, 5}$ is $0$, which means that $\Gamma$ admits a trivial $5$-class number.\\
To finish the proof, we use the results of C. Parry in \cite{Pa}, which states that $5$ divides $h_\Gamma$ if and only if $5^2$ divides $h_k$. The fact that $\Gamma$ admits a trivial $5$-class number, imply that $5^2$ does not divides $h_k$, but we proved that rank $C_{k,5}^{(\sigma)} = 1$, then there is exact divisibility of $h_k$ by $5$.

\section{Numerical examples}
Let $\Gamma\,=\,\mathbb{Q}(\sqrt[5]{n})$ be a pure quintic field and $k$ its normal closure. Using the system PARI/GP \cite{PRI}, we compute the $5$-class number of $\Gamma$ and $k$ for each forms of the natural $n$. The following tables illustrate out main results theorem \ref{theo 1.1}.

\begin{center}

Table 1: $n\,=\, q_1q_2$ with $q_i\,\equiv \,\pm7\, (\mathrm{mod}\, 25)$\\
\begin{tabular}{|c|c|c|c|c|c|c|}
\hline 
$q_1$  & \small$q_1\,(\mathrm{mod}\,25)$& $q_2$ &  \small$q_2\,(\mathrm{mod}\,25)$& $n\,=\,q_1q_2$ & $h_{k,5}$ &   $h_{\Gamma,5}$\\ 
\hline 
7  & 7 & 43  & -7 & 301  & 5  & 1 \\
7  & 7 & 193  & -7 & 1351  & 5  & 1 \\ 
7  & 7 & 293  & -7 & 2051  & 5   & 1 \\ 
107  & 7 & 43  & -7 & 4601  & 5  & 1 \\
157  & 7 & 43 & -7 & 6751  & 5 &  1 \\
457  & 7 & 43 & -7 &   19651  & 5 &  1 \\
107  & 7 & 193  & -7 & 20651  & 5 &  1 \\
557  & 7 & 43 & -7 &   23251  & 5 &  1 \\
607  & 7 & 43 & -7 &   26101  & 5 &  1 \\
157  & 7 & 193  & -7 & 30301 & 5 &  1 \\
107  & 7 & 293  & -7 & 31351  & 5 &  1 \\
757  & 7 & 43  & -7 & 32551  & 5 &  1 \\
857  & 7 & 43  & -7 & 36851 & 5 &  1 \\
907  & 7 & 43  & -7 & 39001  & 5 &  1 \\
107  & 7 & 443  & -7 & 47401  & 5 &  1 \\
257  & 7 & 193  & -7 & 49601  & 5 &  1 \\
307  & 7 & 193  & -7 & 59251  & 5 &  1 \\
157  & 7 & 443  & -7 & 69551  & 5 &  1 \\
257  & 7 & 293  & -7 & 75301  & 5 & 1 \\
457  & 7 & 443  & -7 & 202451 & 5 &  1 \\

\hline 
\end{tabular} 
\end{center}
\newpage
\begin{center}
 Table 2 : $n\,=\, 5q_1$ with $q_1\,\equiv \,\pm7\, (\mathrm{mod}\, 25)$\\
\begin{tabular}{|c|c|c|c|c|}
\hline 
$q_1$ &  $q_1\,(\mathrm{mod}\,25)$ &$n\,=\,5q_1$ & $h_{k,5}$  & $h_{\Gamma,5}$ \\ 
\hline 
7 &7 &     35    & 5 &  1 \\ 
43 &  -7 & 215  &5 &  1 \\ 
107 &  7 & 535   & 5 &  1 \\ 
157 &  7 & 785   & 5 &  1 \\  
193 &  -7& 965   & 5 & 1 \\
257 &  7 & 1285  & 5 &  1 \\ 
293 & -7 & 1465   &  5 &  1 \\ 
307 &  7 & 1535   & 5 &  1 \\ 
443 & -7 & 2215  & 5 &  1 \\ 
457 &  7 & 2285  & 5 &  1 \\ 
557 & 7 &  2785  &  5 &  1 \\ 
607 &  7 & 3053  &5 &  1 \\ 
643 &  -7& 3215   & 5 &  1 \\ 
757 &  7 & 3785  & 5 &  1 \\ 
857 &  7 & 4285  & 5 &  1 \\ 
907 &  7 & 4535  &5 &  1 \\ 
\hline
\end{tabular}
\end{center}
\vspace{0.1cm}
\begin{center}
Table 3: $n\,=\, 5q_1q_2$ with $q_1$ or $q_2$  $\not\equiv \,\pm7\, (\mathrm{mod}\, 25)$\\
\begin{tabular}{|c|c|c|c|c|c|c|}
\hline 
$q_1$ & \small$q_1\,(\mathrm{mod}\,25)$& $q_2$ & \small$q_2\,(\mathrm{mod}\,25)$& $n\,=\,5q_1q_2$ & $h_{k,5}$ &  $h_{\Gamma,5}$\\ 
\hline 
2  & 2 & 3  & 3 &          30  & 5 &  1 \\
7  & 7 & 3 & 3 &         105  & 5 &  1 \\ 
2  & 2 & 13  & 13 &       130  & 5 &  1 \\
2  & 2 & 23  & -2 &       230  & 5 &  1 \\
17  & 17 & 3 & 3 &       255  & 5 &  1 \\
2  & 2 & 53  & 3 &        530  & 5 &  1 \\
37  & 17 & 3 & 3 &       555 & 5 &  1 \\
47  & -3 & 3 & 3 &       705  & 5 & 1 \\
67 & 17 & 3 & 3 &        1005  & 5 &  1 \\
17  & 17 & 23  & -2 &    1955  & 5 &  1\\
37  & 17 & 13  & 13 &    2405 & 5 &  1 \\
47  & -3 & 13  & 13 &    3055  & 5 &  1 \\
47  & -3 & 23  & -2 &    5405  & 5 &  1 \\
47  & -3 & 43  & -7 &    10105  & 5 &  1 \\
107  & 7 & 23  & -2 &    12305  & 5 &  1 \\
67 & 17 & 53 & 3 &       17755  & 5 &  1 \\
97 & -3 & 43 & -7 &      20855  & 5 & 1 \\
\hline 
\end{tabular}
\end{center}

%\section*{\Large References}

\end{document}